\documentclass[final,3p]{elsarticle}

\usepackage{amsmath,amsthm,amssymb,mathrsfs}
\usepackage{enumerate}
\usepackage{slashed}
\usepackage{txfonts,dsfont}
\usepackage{aliascnt}
\usepackage[colorlinks, linkcolor=red, citecolor=green,filecolor=magenta,urlcolor=cyan]{hyperref}

\allowdisplaybreaks

\makeatletter
\newcommand{\autorefcheckize}[1]{%
  \expandafter\let\csname @@\string#1\endcsname#1%
  \expandafter\DeclareRobustCommand\csname relax\string#1\endcsname[1]{%
    \csname @@\string#1\endcsname{##1}\wrtusdrf{##1}}%
  \expandafter\let\expandafter#1\csname relax\string#1\endcsname
}
\makeatother

\allowdisplaybreaks
\newcommand{\abs}[1]{\left\lvert#1\right\rvert}
\newcommand{\kh}[1]{\left(#1\right)}%
\newcommand{\zkh}[1]{\left[#1\right]}%
\newcommand{\dkh}[1]{\left\{#1\right\}}%
\newcommand{\jkh}[1]{\left\langle #1\right\rangle}
\newcommand{\Graph}{\operatorname{Graph}}

\theoremstyle{plain}
\newtheorem{theorem}{Theorem}[section]

\newaliascnt{lem}{theorem}

 \aliascntresetthe{lem}

\newaliascnt{cor}{theorem}

 \aliascntresetthe{cor}

\newaliascnt{prop}{theorem}
\newtheorem{prop}[prop]{Proposition}
 \aliascntresetthe{prop}

\theoremstyle{remark}
\newtheorem{rem}{Remark}[section]

\theoremstyle{definition}

\newtheorem{eg}{Example}[section]

\numberwithin{equation}{section}

\begin{document}

\begin{frontmatter}


\title{A strict upper volume bound for minimal graphs in the unit ball}

\author[cui]{Qing Cui}
\address[cui]{School of Mathematics,
Southwest Jiaotong University, 611756
Chengdu, Sichuan, China}
\ead{cuiqing@swjtu.edu.cn}

\date{\today ~ [120430]}

\begin{abstract}
Let $u$ be a solution of the minimal surface equation on a domain containing the closed unit ball $\overline{B^n}\subset\mathbb R^n$. A classical calibration argument gives $|\Graph_u\cap B^{n+1}| \leq
    \frac12\mathcal |\mathbb S^n|.$ A basic question is whether this half-sphere bound is sharp for minimal graphs. We show that it is not. More precisely, for every
$n\geq2$, there exists an explicit constant $\delta_n>0$, depending only on $n$, such that $ |\Graph_u\cap B^{n+1}|\leq \frac12 |\mathbb S^n|-\delta_n$.

The proof combines calibration with a canonical spherical filling associated with the subgraph and a quantitative incompatibility between near equality in the calibration estimate and the divergence-free structure of the
minimal surface equation. We also give an improved explicit gap and formulate the corresponding sharp extremal problem.
 \end{abstract}

\begin{keyword} minimal graph \sep minimal surface equation \sep calibration \sep upper volume bound \sep quantitative gap

\par
\MSC[2020]   53A10\sep 53C24
\end{keyword}

\end{frontmatter}

\section{Introduction}

Volume estimates are among the most basic quantitative questions in the theory of minimal submanifolds. On the lower-bound side, the monotonicity formula (see e.g. \cite{All72, Sim84, ColMin11}) implies that if a $k$-dimensional minimal submanifold $\Sigma^k\subset B^N$ passes through the origin and
$\partial\Sigma\subset\partial B^N$, then  $|\Sigma|\geq |B^k|.$ Here and in what follows, $B_r^n$ denotes the $n$-dimensional ball of radius $r$  centered at the origin in $\mathbb{R}^n$, and for simplicity, we write the $n$-dimensional unit ball $B^n:=B^n_1$, and the $(n-1)$-dimensional unit sphere $\mathbb S^{n-1}:=\partial B^n$

More refined sharp lower-volume problems have been studied extensively. For instance, Alexander and Osserman  \cite{AleOss75} considered minimal surfaces passing through a prescribed point in the unit ball, and their conjectured sharp lower bound was later established in greater generality by Brendle and Hung \cite{BreHun17}. Sharp lower bounds are also known in several free-boundary and capillary settings; see, for example,
\cite{FraSch11, Bre12, WanXiaZha25} and the references therein.

The situation for upper bounds is quite different. In general, even rather restrictive classes of minimal submanifolds in a fixed ball need not admit a universal upper-volume bound (see \cite{CarSchWiy25}). For minimal graphs, however, calibration provides a distinguished global structure and leads to a universal estimate. Let $\Omega\subset\mathbb R^n$ and let
$u:\Omega\to\mathbb R$ solve the minimal surface equation
\begin{equation}\label{MSE}
    \operatorname{div}\kh{ \frac{\nabla u}{\sqrt{1+|\nabla u|^2}}}=0.
\end{equation}
If $B_r^n\subset\Omega$, the standard calibration argument ({\cite[\S 5.4.18]{Fed69}}, {\cite[Corollary 1.2]{ColMin11}}) gives
\begin{equation}\label{eq:classical-upper}
\abs{\Graph_u\cap B_r^{n+1}}\leq \frac12\abs{\mathbb S^n}\,r^n.
\end{equation}
In particular, an entire minimal graph has Euclidean volume growth, a fact that plays an important role in the classical Bernstein theory.

Estimate \eqref{eq:classical-upper} raises a basic extremal question: is the constant $\frac12\abs{\mathbb S^n}$
sharp within the class of minimal graphs? Equivalently, can one find minimal graphs whose portions inside the unit ball have volume arbitrarily close to that of a hemisphere?

There is a reason why this question is not answered by calibration alone. The half-sphere bound is obtained by comparing the graph with a spherical filling of its boundary. Such an argument uses the area-minimizing property of the graph, but it does not fully exploit
the differential constraint imposed by the minimal surface equation. Writing
\begin{equation*}
q:=\frac{\nabla u}{\sqrt{1+|\nabla u|^2}}
\end{equation*}
equation \eqref{MSE} becomes
\begin{equation}\label{divq}
\operatorname{div}q=0,  \quad |q|<1.
\end{equation}
As we shall see, this divergence-free condition produces a genuine quantitative obstruction to approaching the half-sphere bound.

Our main result shows that the classical upper bound is not merely non-sharp: there is a uniform positive gap depending only on the dimension.

\begin{theorem}\label{thm:main}
Let $n\geq2$, let $\Omega\subset\mathbb R^n$ be open with $\overline{B^n}\subset\Omega$, and let
\(u\in C^2(\Omega)\) satisfy \eqref{MSE}. Then
\begin{equation}\label{eq:main-gap}
\abs{\Graph_u\cap B^{n+1}} \leq \frac12\abs{\mathbb S^n}-\delta_n,
\end{equation}
where $\delta_n=\tau_n^3>0$ and $\tau_n$ is the unique positive root of
\begin{equation}\label{eq:def-taun}
6t^3+4\kh{\abs{6|\mathbb S^{n-1}}^2}^{1/3}t-\frac{n}{n+2}|B^n|=0.
\end{equation}
\end{theorem}

The mechanism behind the proof is simple in principle. Using the subgraph of $u$ inside the unit ball, we construct two canonical spherical fillings of the boundary of $\Graph_u\cap B^{n+1}$. At least one of them has volume at most $\frac12\abs{\mathbb S^n}$ and has a single spherical orientation. If the calibration estimate were nearly an equality, then the upward unit normal of the graph would have to be close, on most of the relevant spherical region, to either the radial field
$X$ or $-X$, where
\begin{equation}\label{xz}
X=(x_1,\cdots, x_n, x_{n+1}):=(x,z), \quad z:=x_{n+1}
\end{equation}
is the position vector. After projection to $B^n$, this forces $q$ to be close in $L^2$ to either $-x$ or $x$. This is impossible beyond a definite quantitative threshold because
\begin{equation}
 \operatorname{div}q=0,\quad \operatorname{div}(\pm x)=\pm n. \label{divqdivx}
\end{equation}
   
This incompatibility is the source of the positive gap in
\eqref{eq:main-gap}. Notice that the argument does not require the intersection of the graph with the unit sphere to be connected or transverse.

The explicit constant $\delta_n$ above is not intended to be optimal. In Appendix~A we refine one step of the argument and obtain a larger explicit constant $\tilde{\delta}_n>\delta_n$. The natural remaining problem is to determine the sharp upper volume. Let $\mathscr M_n$ denote the class of all pairs $(\Omega,u)$ satisfying the assumptions of Theorem \ref{thm:main},
and define
\begin{equation}\label{eq:def-Lambda}
 \Lambda_n:=\sup_{(\Omega,u)\in\mathscr M_n}
  \abs{\Graph_u\cap B^{n+1}}.
\end{equation}
It is therefore natural to ask for the exact value of $\Lambda_n$, whether the supremum is attained by a
smooth minimal graph, and, if it is not attained, what geometric degenerations can occur along maximizing sequences. These questions form a natural sharp extremal problem for the minimal surface equation.
 
\begin{rem}
The flat graph already gives the trivial lower bound $\Lambda_n\ge |B^n|.$ The cylindrical Scherk example in Example \ref{ex:cylindrical-scherk} improves this to an explicit strict lower bound.  Combining Theorem \ref{thm:main}, Proposition \ref{prop:improved-gap} and Example \ref{ex:cylindrical-scherk}, we have the following estimate:
\begin{equation*}
|B^n|+c_n\leq \Lambda_n \leq \frac12|\mathbb S^n|-\tilde{\delta}_n,
\end{equation*}
where $\tilde{\delta}_n$ and $c_n$ are positive constants depending only on $n$ defined in Appendix.
\end{rem}

\vspace{1cm}
\section{Proof of Theorem \ref{thm:main}}
\begin{proof}[Proof of Theorem \ref{thm:main}]
Note that the intersection $\Graph_u\cap \mathbb S^n$ need not be transverse and its induced spherical filling need not have a smooth boundary in the classical sense. Hence a purely smooth application of Stokes' theorem would require additional generic assumptions or an approximation argument.  We will use integral currents to avoid imposing unnecessary regularity assumptions on the spherical intersection. By working with the currents associated with the relevant finite-perimeter sets, the boundary identities, the orientations of the spherical fillings, and their masses are all well defined up to $\mathcal H^n$-null sets. Thus the current language is used only as a technical device to make the calibration argument valid for arbitrary minimal graphs defined over a neighborhood of $B^n$.

Put $\Sigma:=\Graph_u\cap B^{n+1},$
and orient $\Sigma$ by its upward unit normal
\begin{equation*}
  N=\frac{1}{\sqrt{1+\abs{\nabla u}^2}}(-u_{x_1},\cdots, -u_{x_n}, 1)=(-q,s),
    \quad
    q=\frac{\nabla u}{\sqrt{1+|\nabla u|^2}},
    \quad
    s:=\frac{1}{\sqrt{1+|\nabla u|^2}}.
\end{equation*}
Let $(x_1, \cdots, x_n, x_{n+1})$ be the standard coordinates of $\mathbb R^{n+1}$ and we use the notation  $X, x, z$ introduced in   \eqref{xz}. We extend $N$ to $\Omega\times\mathbb R$ by making it independent of $z$. Since by \eqref{divq}
\begin{equation*}
\operatorname{div}_{\mathbb R^{n+1}}N=-\operatorname{div}_{\mathbb R^n}q=0,
\end{equation*}
the $n$-form $\omega:=\iota_N\kh{dx_1\wedge\cdots\wedge dx_n\wedge dz}$ is closed. Moreover, $\omega$ has comass one and calibrates $\Sigma$, so that
\begin{equation}\label{eq:calibration-on-Sigma}
    \omega|_\Sigma=dV_\Sigma.
\end{equation}

We first construct a spherical filling having a single orientation.

Consider the set
\begin{equation*}
    \mathcal E:=\dkh{ (x,z)\in B^{n+1}: \ z<u(x)}.
\end{equation*}

On $\mathbb S^n$, set
\begin{equation*}
    S_- :=\dkh{(x,z)\in \mathbb S^n:\ z<u(x)}, \quad S_+:= \mathbb S^n \setminus S_-,
\end{equation*}
where sets differing by a set of $\mathcal H^n$-measure zero are identified. More precisely, all spherical sets are understood up to $\mathcal H^n$-null sets. This convention is harmless, since the associated currents, masses, and integrals of smooth $n$-forms are unchanged by modifying the sets on $\mathcal H^n$-null subsets.

We use the standard notation $\llbracket M\rrbracket$ for the
current of integration over an oriented rectifiable set $M$. If $M$ is an oriented smooth manifold, then $\llbracket M\rrbracket$ is the current of integration over $M$.
The mass of a $k$-current $T$ is defined by
\begin{equation*}
    \mathbf M(T) :=\sup\dkh{T(\omega): \omega\in C^\infty_0(\Lambda^k\mathbb{R}^{n+1}), \|\omega\|_*\le 1},
\end{equation*}
where $\|\omega\|_*$ denotes the comass norm. In particular, if $M$ is a smooth oriented $k$-dimensional submanifold with multiplicity one, then $\mathbf M(\llbracket M\rrbracket)=\mathcal H^k(M).$  The boundary operator for currents is characterized by Stokes' formula: $\partial T(\eta)=T(d\eta).$ Thus, in the smooth case, $\partial\llbracket M\rrbracket=\llbracket\partial M\rrbracket.$

We orient $\mathbb S^n$ by its outward unit normal $X$, where $X=(x,z)$ denotes the position vector. Up to an $\mathcal H^n$-null set, the reduced boundary of $\mathcal{E}$ consists of the graphical part $\Sigma$ and the spherical part $S_-$. The Gauss--Green formula for sets of finite perimeter (see \cite[Theorem 4.5.6]{Fed69}) gives, with the above orientations,
\begin{equation}\label{eq:boundary-E}
\partial\llbracket\mathcal E\rrbracket= \llbracket\Sigma\rrbracket+\llbracket S_-\rrbracket.
\end{equation}
Consequently,
\begin{equation*}
    \partial\llbracket\Sigma\rrbracket=-\partial\llbracket S_-\rrbracket=\partial\llbracket S_+\rrbracket.
\end{equation*}
It follows that both
\begin{equation*}
T_-:=-\llbracket S_-\rrbracket\quad\text{and}
\quad 
T_+:=\llbracket S_+\rrbracket
\end{equation*}
have boundary $\partial\llbracket\Sigma\rrbracket$, and  masses are $ \mathbf M(T_-)=\mathcal H^n(S_-),
    \  \mathbf M(T_+)=\mathcal H^n(S_+).$
Hence
\begin{equation*}
    \mathbf M(T_-)+\mathbf M(T_+)=
    \mathcal H^n(S_-)+\mathcal H^n(S_+)=\mathcal H^n(\mathbb S^n).
\end{equation*}
We may therefore choose one of these two currents, denoted by $T$, such that
\begin{equation}\label{eq:T-mass-half}
\partial T=\partial\llbracket\Sigma\rrbracket,
    \quad
\mathbf M(T)\leq\frac12\mathcal H^n(\mathbb S^n).
\end{equation}
More precisely, there exist a measurable region
$S\subset\mathbb S^n$ and a constant
$\sigma\in \dkh{1,-1}$ such that $T=\sigma\llbracket S\rrbracket$ and the unit normal determining the orientation of $T$ is $\nu_T=\sigma X$ a.e. on $S$.

Throughout the paper, $|M|$ denotes the Hausdorff measure of the dimension appropriate to $M$. In particular,
\begin{equation*}
|B^m|=\mathcal L^m(B^m), \quad |\mathbb S^m|=\mathcal H^m(\mathbb S^m).
\end{equation*}


Since $d\omega=0$, we have
\begin{equation*}
    0=\llbracket E\rrbracket(d\omega) =\partial\llbracket E\rrbracket(\omega).
\end{equation*}
Using \eqref{eq:boundary-E}, this gives
\begin{equation*}
\int_\Sigma\omega =  -\int_{S_-}\omega  =  \int_{T_-}\omega. 
\end{equation*}
On the other hand,
\begin{equation*}
\llbracket S_-\rrbracket+\llbracket S_+\rrbracket = \llbracket \mathbb S^n\rrbracket = \partial\llbracket B^{n+1}\rrbracket, 
\end{equation*}
and hence
\begin{equation*}
\int_{S_-}\omega+\int_{S_+}\omega =  \llbracket B^{n+1}\rrbracket(d\omega)  = 0. 
\end{equation*}
Therefore
\begin{equation*}
\int_\Sigma\omega = \int_{T_-}\omega = \int_{T_+}\omega. 
\end{equation*} 
Thus, for the chosen current $T$, we have
\begin{equation}\label{eq:calibration-filling}
\abs{\Sigma}=\int_\Sigma\omega =\int_T\omega =\int_S\langle N,\sigma X\rangle\,dV\le \int_S |N|  \ |X|dV =|S|.
\end{equation}
Therefore, by \eqref{eq:T-mass-half} and \eqref{eq:calibration-filling}, we have 
\begin{equation}\label{eq:E-D-eta}
|\Sigma| \le |S| \le \frac12 \abs{\mathbb S^n}.
\end{equation}

%

We first consider the case $\sigma=1$. In this case, by \eqref{eq:calibration-filling},
\begin{equation}\label{eq:D-def}
\abs{S}-\abs{\Sigma}=\int_S\kh{1-\jkh{N, X}}dV.
\end{equation}
Since the equator $\mathbb S^n\cap\{z=0\}\) has
$\mathcal H^n$-measure zero, we have
\begin{equation}\label{Sdecomp}
    |S|=|S\cap \dkh{z>0}|+|S\cap \dkh{z<0}|.
\end{equation}
Moreover,
\begin{align*}
    \mathbb S^n \cap \dkh{z>0}=&(( \mathbb S^n \cap \dkh{z>0}) \cap S)\,\cup\, ((\mathbb S^n \cap \dkh{z>0})\setminus S)
=(S\cap\dkh{z>0})\,\cup\, ((\mathbb S^n \cap \dkh{z>0})\setminus S),
\end{align*}
and these two sets are disjoint. Hence
\begin{equation*}
    |S\cap \dkh{z>0}|=\abs{ \mathbb S^n \cap \dkh{z>0}}-\abs{(\mathbb S^n \cap \dkh{z>0})\setminus S}.
\end{equation*}
Therefore, by \eqref{Sdecomp}, we have
\begin{equation}\label{Sdecomp-volume}
\begin{aligned}
|S|=&\abs{ \mathbb S^n \cap \dkh{z>0}}-\abs{(\mathbb S^n \cap \dkh{z>0})\setminus S}+\abs{S\cap\dkh{z<0}}
\\
=&\frac12\abs{\mathbb S^n}-\abs{(\mathbb S^n \cap \dkh{z>0})\setminus S}+\abs{S\cap\dkh{z<0}}.
\end{aligned}
\end{equation}
Inequality \eqref{eq:E-D-eta} gives
\begin{equation}\label{m-l}
\abs{(\mathbb S^n \cap \dkh{z>0})\setminus S}-\abs{S\cap\dkh{z<0}}=\frac12\abs{\mathbb S^n}-|S|\ge0.
\end{equation}

Since $\sigma=1$, near equality of \eqref{eq:calibration-filling} would force $N$ to be close to the position vector $X$. As $N_{n+1}>0$, the unfavorable part of $S$ is its intersection with the lower hemisphere $\mathbb S^n\cap \dkh{z<0}$. We now estimate this part. Let $X=(x,z)\in \mathbb S^n\cap \dkh{z<0}$. Among all unit vectors $Y\in\mathbb R^{n+1}$ satisfying $Y_{n+1}\geq0$, one has
\begin{equation*}
\jkh{X, Y}=\sum_{i=1}^n Y_i x_i +Y_{n+1}z\le \sum_{i=1}^n Y_i x_i\le \sqrt{\sum_{i=1}^n x_i^2}=\sqrt{1-z^2}.
\end{equation*}
Since the last component of $N$ is positive, it follows that
\begin{equation*}
1-\jkh{N,X} \geq 1-\sqrt{1-z^2} \geq \frac{z^2}{2},
    \quad \text{for}\ \ X=(x,z)\in \mathbb S^n\cap \dkh{z<0}.
\end{equation*}
Consequently, by \eqref{eq:D-def},
\begin{equation}\label{eq:D-lower-z}
    |S|-|\Sigma|\geq \frac12 \int_{S\cap \dkh{z<0}}z^2\,dV.
\end{equation}

We use the following spherical parametrization on $\mathbb S^n$:
\begin{equation*}
\Phi(\theta,z)=\kh{\sqrt{1-z^2}\,\theta, z}, \quad \theta\in\mathbb S^{n-1},\quad -1<z<1.
\end{equation*}
Writing $r(z)=\sqrt{1-z^2}$, the induced metric is
\begin{equation*}
g_{\mathbb S^n}=r(z)^2 g_{\mathbb S^{n-1}}+\kh{1+r'(z)^2}\,dz^2 =(1-z^2)g_{\mathbb{S}^{n-1}} + \frac{1}{1-z^2} dz^2.
\end{equation*}
Hence in this parametrization, we have
\begin{equation}\label{dVSn}
dV_{\mathbb S^n} = (1-z^2)^{\frac{n-2}{2}}\,d\theta\,dz.
\end{equation}
Thus, when integrating over a measurable subset of \(\mathbb S^n\),
we always mean the restriction of this spherical measure. Using \eqref{dVSn},   for $0<a<1$, 
\begin{equation}\label{eq:slab-estimate}
\abs{ \mathbb S^n\cap \dkh{-a<z<0}} = \abs{\mathbb{S}^{n-1}}\int_0^a (1-t^2)^{\frac{n-2}{2}}\,dt \leq \abs{\mathbb{S}^{n-1}} a.
\end{equation}
It follows that, for $0<t<1$,
\begin{equation*}
\abs{S\cap \dkh{z<-t}} =\abs{S\cap \dkh{z<0}}-\abs{S\cap \dkh{-t < z <0}} \geq \max\dkh{\abs{S\cap\dkh{z<0}}-\abs{\mathbb{S}^{n-1}}t, \ \ 0}.
\end{equation*}
Since $0\leq |z|\leq1$ on $S\cap \dkh{z<0}$, Tonelli's theorem \cite[Theorem 2.37(a)]{Fol99} gives
\begin{equation}
\begin{aligned}
\frac12\int_{S\cap \dkh{z<0}}z^2\,dV=&\int_{S\cap \dkh{z<0}} \kh{\int_0^{|z|} t dt} dV\\
=& \int_{S\cap\dkh{z<0}}
    \int_0^1 t\,\mathbf 1_{\{|z|>t\}}\,dt\,dV  
\\               
=&\int_0^1 t \kh{\int_{S\cap\dkh{z<0}} \mathbf{1}_{\dkh{|z|>t}} dV  }dt
\\
=&\int_0^1 t\, \abs{S\cap \dkh{z<-t}}\,dt.
\end{aligned}\label{eq:Tonelli}
\end{equation}
Note that 
\begin{equation*}
\frac12\abs{\mathbb S^n} =\abs{\mathbb S^{n-1}} \int_0^1 (1-t^2)^{\frac{n-2}{2}} dt \leq \abs{\mathbb S^{n-1}}.
\end{equation*}
Thus, \eqref{eq:E-D-eta},  \eqref{eq:slab-estimate} and \eqref{eq:Tonelli} give 
\begin{equation}\label{eq:ell-cubic}
\begin{aligned}
\frac12 \int_{S\cap \dkh{z<0}} z^2 dV 
\geq&
 \int_0^{\frac{\abs{S\cap\dkh{z<0}}}{|\mathbb{S}^{n-1}|}} t\, \abs{S\cap \dkh{z<-t}}\,dt
\\
=& \int_0^{\frac{\abs{S\cap\dkh{z<0}}}{|\mathbb{S}^{n-1}|}} t\, \kh{\abs{S\cap \dkh{z<0}}-\abs{S\cap\dkh{-t<z<0}}}\,dt
\\
\ge& \int_0^{\frac{\abs{S\cap\dkh{z<0}}}{|\mathbb{S}^{n-1}|}} t\, \kh{\abs{S\cap \dkh{z<0}}-t \abs{\mathbb S^{n-1}}}\,dt
\\
=&   \frac{\abs{S\cap \dkh{z<0}}^3}{6\abs{\mathbb S^{n-1}}^2}.
\end{aligned}
\end{equation}
 
Combining \eqref{eq:D-lower-z}, \eqref{eq:ell-cubic}, and
\eqref{eq:E-D-eta}, we find
\begin{equation}\label{eq:ell-bound}
\abs{S\cap \dkh{z<0}}\leq \kh{6\abs{\mathbb S^{n-1}}^2 \kh{\abs{S}-\abs{\Sigma}}}^{1/3} \leq \kh{6\abs{\mathbb S^{n-1}}^2 \kh{\frac12 \abs{\mathbb S^n}-\abs{\Sigma}}}^{1/3} 
\end{equation}
Thus, by \eqref{m-l} and \eqref{eq:ell-bound},
\begin{equation}\label{eq:m-bound}
    \abs{\kh{\mathbb S^n\cap\dkh{z>0}}\setminus S}=\abs{S\cap\dkh{z<0}} +\frac12\abs{\mathbb S^n}-\abs{S}\leq  \kh{6\abs{\mathbb S^{n-1}}^2 \kh{\frac12 \abs{\mathbb S^n}-\abs{\Sigma}}}^{1/3}  +\frac12\abs{\mathbb S^n} -\abs{\Sigma}.
\end{equation}

We next project the part of $S$ lying in the upper hemisphere onto $B^n$. Define
\begin{equation*}
P:=\dkh{x\in B^n: \kh{x,\sqrt{1-|x|^2}}\in S}.
\end{equation*}
Since the area element of the upper hemisphere satisfies
\begin{equation*}
dV_{\mathbb S^n}=\frac{dx}{\sqrt{1-|x|^2}} \geq dx,
\end{equation*}
we have
\begin{equation}\label{eq:missing-projection}
    |B^n\setminus P|=\int_{B^n\setminus P} 1 dx \leq \int_{B^n\setminus P} \frac{1}{\sqrt{1-\abs{x}^2}} dx =   \abs{\kh{\mathbb S^n\cap\dkh{z>0}}\setminus S}
\end{equation}

For $X=\kh{x, \sqrt{1-\abs{x}^2}}\in S\cap \dkh{z>0}$,
we have
\begin{equation*}
\abs{q+x}^2\leq \abs{-q-x}^2 +(s-z)^2=\abs{(-q, s)-(x, z)}^2=\abs{N-X}^2=2\kh{1-\jkh{N, X}}.
\end{equation*}
Therefore, by \eqref{eq:D-def},
\begin{equation}\label{eq:q-plus-x-on-P}
\int_P|q+x|^2\,dx \leq 2\int_{S\cap \dkh{z>0}}\kh{1-\jkh{N, X}} \,dV                              
\leq 2\kh{\abs{S}-\abs{\Sigma}}.
\end{equation}
On $B^n\setminus P$, we use
\begin{equation}
    |q+x| \leq |q|+|x| \leq 2. \label{q+x}
\end{equation}
By \eqref{eq:missing-projection} and \eqref{eq:m-bound},
\begin{equation}
\begin{aligned}
\int_{B^n}|q+x|^2\,dx =&\int_P |q+x|^2 \, dx+ \int_{B^n\setminus P} |q+x|^2\\
\leq &2\kh{\abs{S}-\abs{\Sigma}} +4\abs{B^n\setminus P}\\
\leq &2\kh{\abs{S}-\abs{\Sigma}} +4\abs{\mathbb S^n\cap \dkh{z>0}\setminus S}\\
\leq & 2\kh{\abs{S}-\abs{\Sigma}} +4\zkh{\kh{6\abs{\mathbb S^{n-1}}^2 \kh{\frac12 \abs{\mathbb S^n}-\abs{\Sigma}}}^{1/3}  +\frac12\abs{\mathbb S^n} -\abs{\Sigma}}\\
\leq & 4\kh{6\abs{\mathbb S^{n-1}}^2 \kh{\frac12 \abs{\mathbb S^n}-\abs{\Sigma}}}^{1/3} +6\kh{\frac12\abs{\mathbb S^n} -\abs{\Sigma}}.
\end{aligned}
\label{eq:q-plus-x-upper}
\end{equation}

We now derive a lower bound using the divergence-free condition \eqref{divq}.  Let $\phi(x):=\frac{1-|x|^2}{2n}.$ Then $\phi=0$ on $\mathbb S^{n-1}$. Using \eqref{divqdivx}, integration by parts gives
\begin{equation*}
n\int_{B^n}\phi\,dx =\int_{B^n} \operatorname{div} (q+x) \phi \, dx= -\int_{B^n} (q+x)\cdot\nabla\phi\,dx.
\end{equation*}
Hence, by the Cauchy--Schwarz inequality,
\begin{equation}\label{eq:div-cauchy}
\kh{n\int_{B^n}\phi\,dx}^2 \leq\kh{\int_{B^n}|q+x|^2\,dx}
\kh{\int_{B^n}|\nabla\phi|^2\,dx}.
\end{equation}
A direct computation yields
\begin{equation*}
 \int_{B^n}\phi\,dx= \int_{B^n}|\nabla\phi|^2\,dx
=\frac{\abs{B^n}}{n(n+2)}.
\end{equation*}
Therefore, by \eqref{eq:div-cauchy},
\begin{equation}\label{eq:q-plus-x-lower}
\int_{B^n}|q+x|^2\,dx\geq\frac{n \abs{B^n}}{n+2}.
\end{equation}
Combining \eqref{eq:q-plus-x-upper} and
\eqref{eq:q-plus-x-lower}, we obtain
\begin{equation}\label{eq:E-inequality}
    \frac{n \abs{B^n}}{n+2} \leq  4\kh{6\abs{\mathbb S^{n-1}}^2 \kh{\frac12 \abs{\mathbb S^n}-\abs{\Sigma}}}^{1/3} +6\kh{\frac12\abs{\mathbb S^n} -\abs{\Sigma}} .
\end{equation}
Set $t=\kh{\frac12\abs{\mathbb S^n} -\abs{\Sigma}}^{1/3}$. Since the function
\begin{equation*}
t\longmapsto 6t^3+4\kh{6\abs{\mathbb S^{n-1}}^2}^{1/3}t
\end{equation*}
is strictly increasing on $[0,\infty)$, equations
\eqref{eq:def-taun} and \eqref{eq:E-inequality} imply
$t\geq\tau_n.$ Consequently, 
\begin{equation*}
\abs{\Sigma} \leq \frac12 \abs{\mathbb S^n} -\tau_n^3.
\end{equation*}

It remains to consider the case $\sigma=-1$. In this case,
the lower hemisphere is now the favorable hemisphere. Repeating the preceding argument with $\mathbb S^n \cap \dkh{z>0}$ and $\mathbb S^n \cap \dkh{z<0}$ interchanged gives
\begin{equation*}
\int_{B^n}|q-x|^2\,dx \leq 6\kh{\frac12\abs{\mathbb S^n}-\abs{\Sigma}}+4\kh{6\abs{\mathbb S^{n-1}}^2 \kh{\frac12\abs{\mathbb S^n}-\abs{\Sigma}}}^{1/3}.
\end{equation*}
Using the same test function $\phi$ and taking absolute values in the  integration-by-parts identity yields
\begin{equation*}
\int_{B^n}|q-x|^2\,dx \geq \frac{n }{n+2} \abs{B^n}.
\end{equation*}
Thus \eqref{eq:E-inequality} holds in this case as well, and again
\begin{equation*}
\frac12\abs{\mathbb S^n} -\abs{\Sigma} \geq \tau_n^3.
\end{equation*}
This proves \eqref{eq:main-gap}.
\end{proof}

\vspace{2cm}
\section*{Appendix A. An improved explicit gap and Scherk examples}
\setcounter{section}{1}
\setcounter{theorem}{0}
\renewcommand{\thesection}{A}
\renewcommand{\thetheorem}{A.\arabic{theorem}}
\subsection{An improved explicit gap}
The explicit constant obtained  in Theorem~\ref{thm:main} can be improved by using the spherical area element more efficiently on the missing part of the projection. 

\begin{prop}
\label{prop:improved-gap}
Under the assumptions of Theorem~\ref{thm:main}, one has
\begin{equation}\label{eq:improved-gap}
\abs{ \Graph_u\cap B^{n+1}} \leq \frac12\abs{\mathbb S^n}-\widetilde\delta_n,
\end{equation}
where $\tilde\delta_n=\tilde\tau_n^3>\delta_n$, and $\tilde\tau_n$ is the unique positive root of
\begin{equation}\label{eq:def-tilde-taun}
\kh{2+\frac{25\sqrt5}{27}}  t^3+ \frac{25\sqrt5}{27}\kh{6\abs{\mathbb S^{n-1}}^2}^{1/3} t -\frac{n}{n+2}\abs{B^n}=0,
\end{equation}
\end{prop}

\begin{proof}
We retain all the notation introduced in the proof of
Theorem~\ref{thm:main}.  The argument agrees with the proof of Theorem \ref{thm:main} up to \eqref{q+x}. We modify the inequality \eqref{q+x} as follows:
\begin{equation*}
|q+x| \leq |q| +|x| \leq 1+|x|.
\end{equation*}
Note that 
\begin{equation*}
(1+t)^2 \sqrt{1-t^2}\leq \frac{25\sqrt{5}}{27}, \quad \forall t\in [0,1].
\end{equation*}
Thus we have
\begin{equation}
\label{eq:Q-part-improved}
\int_{B^n\setminus P}|q+x|^2\,dx\leq \int_{B^n\setminus P}(1+|x|)^2\,dx=\int_{B^n\setminus P} (1+|x|)^2\sqrt{1-|x|^2} \frac{dx}{\sqrt{1-|x|^2}} \leq \frac{25\sqrt{5}}{27}\abs{\mathbb S^n\cap\dkh{z>0}\setminus S}.
\end{equation}
Combining \eqref{eq:D-def},
\eqref{eq:Q-part-improved}, and \eqref{eq:m-bound}, we obtain
\begin{equation}\label{improved-q+x}
\int_{B^n}|q+x|^2\,dx \leq \kh{2+\frac{25\sqrt{5}}{27}} \kh{\frac12 \abs{\mathbb S^n}-\abs{\Sigma}} +\frac{25\sqrt{5}}{27} \kh{6\abs{\mathbb S^{n-1}}^2 \kh{\frac12 \abs{\mathbb S^n}-\abs{\Sigma}} }^{1/3}.
\end{equation}
 
Let  $t=\kh{\frac12 \abs{\mathbb S^n}-\abs{\Sigma}}^{1/3}$. Combining \eqref{eq:q-plus-x-lower} and \eqref{improved-q+x}, we have 
\begin{equation*}
\kh{2+\frac{25\sqrt5}{27}}  t^3+ \frac{25\sqrt5}{27}\kh{6\abs{\mathbb S^{n-1}}^2}^{1/3} t \geq \frac{n}{n+2}\abs{B^n}.
\end{equation*}

Now we need to compare $\tau_n$ and $\tilde{\tau}_n$.
Since $\frac{25\sqrt5}{27}<4$ and   $2+\frac{25\sqrt5}{27}<6$, we have, for every $t>0$,
\begin{align*}
\kh{2+\frac{25\sqrt5}{27}} t^3+\frac{25\sqrt5}{27}\kh{6\abs{\mathbb S^{n-1}}^2}^{1/3}t< 6t^3+4\kh{6\abs{\mathbb S^{n-1}}^2}^{1/3}t.
\end{align*}
The two functions are strictly increasing on $[0,\infty)$.
Comparing their positive solutions at the common level
$n /(n+2)\abs{B^n}$, we obtain $\tilde\tau_n>\tau_n$.

The remainder of the proof is the same as the proof of Theorem \ref{thm:main}.
\end{proof}

\begin{rem}
\label{rem:improved-gap-dimension-two}
When $n=2$, $\tilde\tau_2$ is the unique positive solution of
\begin{equation*}
 \kh{2+\frac{25\sqrt5}{27}} t^3+\frac{25\sqrt5}{27} (24\pi^2)^{1/3} t-\frac{\pi}{2}=0.
\end{equation*}
Numerically, $\tilde\tau_2 \approx 0.1220406592$ and
\begin{equation*}
 \tilde\delta_2=\tilde\tau_2^3\approx 1.8176641\times10^{-3}
>  2.5491736\times10^{-4} \approx\delta_2.
\end{equation*}
\end{rem}

\subsection{A cylindrical Scherk example}

\begin{eg}
\label{ex:cylindrical-scherk}
Consider
\begin{equation*}
v(x_1, x_2):=2\log \frac{\cos(x_2/2)}{\cos(x_1/2)}, \quad (x_1, x_2)\in (-\pi, \pi)^2\supset \overline{B^2}.
\end{equation*}
Clearly, $v$ satisfies the minimal surface equation and $\Graph_v$ is a Scherk type surface. For $n\geq2$, define its cylindrical extension
\begin{equation*}
u(x_1,\cdots,x_n):=v(x_1,x_2), \quad (x_1,\cdots,x_n)\in (-\pi,\pi)^2\times\mathbb R^{n-2} \supset \overline{B^n}.
\end{equation*}
 We claim that there exists an explicit positive constant $\kappa_0$ such that
\begin{equation}\label{eq:fixed-scherk-radial-lower}
 |\Graph_v\cap B_r^3|\geq \pi r^2+\frac{\pi\kappa_0}{4}r^4, \quad 0<r\leq1.
\end{equation}
One may take
\begin{equation}\label{eq:kappa0-fixed}
\kappa_0:=\frac{\beta}{2}-\frac{\tan^2(1/2)}{2} -\frac{\beta\tan^2(1/2)}{8},
\quad \beta:=\frac{1}{\sqrt{1+2\tan^2(1/2)}+1}.
\end{equation}
Numerically, $ \kappa_0\approx 0.0551757.$

We now prove \eqref{eq:fixed-scherk-radial-lower}. For $\mu\geq2$, set 
\begin{equation}\label{def-vmu}
v_\mu(x_1,x_2):= \mu\log\frac{\cos(x_2/\mu)}{\cos(x_1/\mu)}
\end{equation}
and denote $A_\mu:=
    \abs{\Graph{v_\mu}\cap B^3}.$
We first establish the estimate
\begin{equation}\label{eq:scherk-unit-lower}
A_\mu\geq \pi \kh{1+\frac{\kappa_0}{\mu^2}},\quad \mu\geq2,
\end{equation}
where $\kappa_0$ is given by \eqref{eq:kappa0-fixed}.

Write
\begin{equation*}
 x_1=r\cos\theta, \quad x_2=r\sin\theta, \quad 0\leq r\leq1.
\end{equation*}
We first derive an upper bound for the height of the Scherk graph. Define
\begin{equation*}
F(s):=\log\cos\sqrt{s},\quad 0\leq s\leq\frac14.
\end{equation*}
For $s>0$,
\begin{equation*}
F'(s) = -\frac{\tan\sqrt{s}}{2\sqrt{s}},
\end{equation*}
with the limiting value $F'(0)=-1/2$. Moreover, the function $ t\longmapsto\frac{\tan t}{t}$
is increasing on $(0,\pi/2)$. Since $\mu\geq2$, we have $1/\mu\leq1/2$. Thus, by the mean value theorem,
\begin{equation*}
 \abs{v_\mu(r\cos\theta,r\sin\theta)}=\mu \abs{F\kh{\frac{r^2\sin^2\theta}{\mu^2}} - F\kh{\frac{r^2\cos^2\theta}{\mu^2}}}\leq\mu \frac{\mu}{2}\tan\frac1\mu\frac{r^2}{\mu^2}|\sin^2\theta-\cos^2\theta|=\frac12\tan\frac1\mu\,r^2|\cos2\theta|.
\end{equation*}
Since $t\mapsto \tan t/t$ is increasing and $1/\mu\leq1/2$, $\frac{\tan(1/\mu)}{1/\mu}\leq\frac{\tan(1/2)}{1/2},$ and hence
\begin{equation}\label{eq:scherk-height-bound}
\abs{ v_\mu(r\cos\theta,r\sin\theta)}\leq\frac{\tan(1/2)}{\mu} r^2|\cos2\theta|.
\end{equation}

Consider the planar region
\begin{equation}\label{eq:def-Dmu}
D_\mu:=\dkh{(r,\theta): 0\leq r^2\leq 1-\kh{\frac{\tan(1/2)}{\mu}}^2\cos^22\theta}.
\end{equation}
We claim that the graph of $v_\mu$ over $D_\mu$ is contained in $B^3$. Indeed, set
\begin{equation}\label{def-b}
    b:=\kh{\frac{\tan(1/2)}{\mu}}^2\cos^22\theta.
\end{equation}
Since $\mu\geq2$, we have $0\leq b<1$. If $r^2\leq1-b$, then \eqref{eq:scherk-height-bound} gives $v_\mu(r\cos\theta,r\sin\theta)^2 \leq br^4.$
Since the function $s\mapsto s+bs^2$ is increasing for
$s\geq0$,
\begin{equation*}
r^2+v_\mu^2 \leq r^2+br^4 \leq (1-b)+b(1-b)^2 =
    1-2b^2+b^3 \leq1.
\end{equation*}
Thus
\begin{equation}\label{eq:graph-Dmu-contained}
\dkh{(x_1,x_2,v_\mu(x_1,x_2)): (x_1,x_2)\in D_\mu} \subset
\Graph{v_\mu}\cap B^3.
\end{equation}

Next we estimate the area element from below.   Using $|\tan t|\geq |t|$ for $|t|<\pi/2$, we obtain
\begin{equation}\label{eq:scherk-gradient-lower}
  |\nabla v_\mu|^2 = \tan^2(x_1/\mu) +  \tan^2(x_2/\mu) \geq \frac{x_1^2+x_2^2}{\mu^2} =\frac{r^2}{\mu^2}.
\end{equation}
On the other hand, since $r\leq1$ and $\mu\geq2$, $ |\nabla v_\mu|^2 \leq  2\tan^2(1/2).$ Since $t\mapsto\sqrt{1+t}$ is concave on $[0,\infty)$, its graph lies above the chord joining the points $0$ and $2\tan^2(1/2)$. Therefore, for $0\leq t\leq 2\tan^2(1/2)$, we have $\sqrt{1+t} \geq 1+\beta t, $ where $\beta$ is defined in \eqref{eq:kappa0-fixed}. Combining this with \eqref{eq:scherk-gradient-lower}, we obtain
\begin{equation}\label{eq:scherk-area-element-lower}
  \sqrt{1+|\nabla v_\mu|^2}\geq1+\frac{\beta r^2}{\mu^2}\quad\text{on }D_\mu.
\end{equation}

By \eqref{eq:graph-Dmu-contained} and
\eqref{eq:scherk-area-element-lower},
\begin{equation}
    A_\mu\geq \int_{D_\mu}  \sqrt{1+|\nabla v_\mu|^2}\,dx_1dx_2 \geq  |D_\mu| + \frac{\beta}{\mu^2} \int_{D_\mu}r^2\,dx_1dx_2.
    \label{eq:A-mu-lower}
\end{equation} 

We now compute the two integrals explicitly. From
\eqref{eq:def-Dmu} and \eqref{def-b},
\begin{equation}\label{eq:Dmu-area}
 |D_\mu |= \int_0^{2\pi} \int_0^{\sqrt{1-b}}  \,r\,dr\,d\theta = \frac12 \int_0^{2\pi}
\kh{1-b} d\theta=\pi\zkh{1-\frac12  \kh{\frac{\tan(1/2)}{\mu}}^2}.
\end{equation}
Similarly,
\begin{equation}\label{eq:Dmu-r2}
 \int_{D_\mu}r^2\,dx_1dx_2=\int_0^{2\pi} \int_0^{\sqrt{1-b}}  \,r^3\,dr\,d\theta=\frac{\pi}{2}-\frac{\pi}{2} \kh{\frac{\tan(1/2)}{\mu}}^2+\frac{3\pi}{16}\kh{\frac{\tan(1/2)}{\mu}}^4.
\end{equation}
Substituting \eqref{eq:Dmu-area} and \eqref{eq:Dmu-r2} into \eqref{eq:A-mu-lower}, since $\mu\geq 2$, we obtain
\begin{align*}
\frac{A_\mu}{\pi}\geq &1+\frac{1}{\mu^2}\zkh{\frac{\beta}{2}-\frac{ \tan^2(1/2) }{2}-\frac{\beta  \tan^2(1/2) }{2\mu^2}+\frac{3\beta  \tan^4(1/2)}{16\mu^4}}\\
\geq& 1+\frac{1}{\mu^2}\zkh{\frac{\beta}{2}-\frac{ \tan^2(1/2) }{2}-\frac{\beta  \tan^2(1/2) }{2\mu^2}}\\
\geq& 1+\frac{1}{\mu^2}\zkh{\frac{\beta}{2}-\frac{ \tan^2(1/2) }{2}-\frac{\beta  \tan^2(1/2) }{8}}\\
=& 1+\frac{\kappa_0}{\mu^2},
\end{align*}
which proves \eqref{eq:scherk-unit-lower}.

For completeness, let us also verify that $\kappa_0>0$. Since $\sin(1/2)<\frac12, \ \cos(1/2)>1-\frac{(1/2)^2}{2}=\frac78, $ we have $\tan(1/2)<\frac47.$ Consequently, 
$
\sqrt{1+2\tan^2(1/2)} < \sqrt{1+\frac{32}{49}} = \frac97,
$ 
and therefore $\frac7{16}<\beta<\frac12.$ It follows that
\begin{equation*}
\kappa_0=\frac{\beta}{2}-\frac{\tan^2(1/2)}{2} -\frac{\beta\tan^2(1/2)}{8}>\frac{7}{32}-\frac{8}{49}-\frac{1}{49}=\frac{55}{1568}>0.
\end{equation*}
Numerically, $\kappa_0\approx0.0551757.$

We finally deduce \eqref{eq:fixed-scherk-radial-lower}. Recall that, by the definition \eqref{def-vmu}, $v=v_2$.
For $0<r\leq1$, let $\mu=2/r$. A direct computation gives $v_2(rx_1,rx_2)= r\,v_{2/r}(x_1,x_2).$ Hence dilation by the factor $1/r$ maps $\Graph_{v_2}\cap B_r^3$ onto $\Graph_{v_{2/r}}\cap B^3.$
 
Since two-dimensional area scales by the square of the dilation factor,
\begin{equation*}
\abs{ \Graph_v\cap B_r^3}= r^2 \abs{\Graph_{v_{2/r}}\cap B^3}
\end{equation*}
As $2/r\geq2$, estimate \eqref{eq:scherk-unit-lower} yields
\begin{equation*}
 \abs{\Graph_v\cap B_r^3} \geq  r^2\pi
    \kh{ 1+ \frac{\kappa_0}{(2/r)^2}} = \pi r^2 + \frac{\pi\kappa_0}{4}r^4.
\end{equation*}
This proves \eqref{eq:fixed-scherk-radial-lower}.
 
For $n=2$, the estimate follows directly from \eqref{eq:scherk-unit-lower} with $r=1$. Hence, in the slicing argument below, we assume $n\geq3$.
Now write $ y=(x_3,\ldots,x_n)\in\mathbb R^{n-2}.$  By slicing the cylindrical graph, we obtain
\begin{equation*}
 |\Graph_u\cap B^{n+1}| =\int_{B^{n-2}}\abs{\Graph_v \cap B^3_{\sqrt{1-|y|^2}}}\,dy \geq \pi \int_{B^{n-2}}
 (1-|y|^2)\,dy+ \frac{\pi\kappa_0}{4}\int_{B^{n-2}}(1-|y|^2)^2\,dy.
\end{equation*}
 
The first term equals $|B^n|$. Moreover,
\begin{equation*}
  \frac{\pi}{4} \int_{B^{n-2}} (1-|y|^2)^2\,dy= \frac{|B^n|}{n+2}.
\end{equation*}
Therefore
\begin{equation*}
 |\Graph_u\cap B^{n+1}|\geq |B^n|\kh{1+\frac{\kappa_0}{n+2}}=:|B^n|+c_n.
\end{equation*}
\end{eg}

\hspace{2cm}

\begin{thebibliography}{10}

 

\bibitem{AleOss75}
H.~Alexander and R.~Osserman, \emph{Area bounds for various classes of
  surfaces}, American Journal of Mathematics \textbf{97} (1975), 753--769.

\bibitem{All72}
William~K. Allard, \emph{On the first variation of a varifold}, Ann. of Math.
  (2) \textbf{95} (1972), 417--491. \MR{307015}

\bibitem{Bre12}
Simon Brendle, \emph{A sharp bound for the area of minimal surfaces in the unit
  ball}, Geom. Funct. Anal. \textbf{22} (2012), no.~3, 621--626. \MR{2972603}

\bibitem{BreHun17}
Simon Brendle and P.~K. Hung, \emph{Area bounds for minimal surfaces that pass
  through a prescribed point in a ball}, Geometric and Functional Analysis
  \textbf{27} (2017), 235--239.

\bibitem{CarSchWiy25}
Alessandro Carlotto, Mario~B. Schulz, and David Wiygul, \emph{Disc stackings
  and their {Morse} index}, Advanced Nonlinear Studies \textbf{25} (2025),
  no.~3, 756--806.

\bibitem{ColMin11}
Tobias~H. Colding and William~P. Minicozzi, II, \emph{A course in minimal
  surfaces}, Graduate Studies in Mathematics, vol. 121, American Mathematical
  Society, Providence, RI, 2011. \MR{2780140}

\bibitem{Fed69}
Herbert Federer, \emph{Geometric measure theory}, Die Grundlehren der
  mathematischen Wissenschaften, Band 153, Springer-Verlag New York, Inc., New
  York, 1969. \MR{257325}

\bibitem{Fol99}
Gerald~B. Folland, \emph{Real analysis: Modern techniques and their
  applications}, 2 ed., Pure and Applied Mathematics, John Wiley \& Sons, New
  York, 1999.

\bibitem{FraSch11}
Ailana Fraser and Richard Schoen, \emph{The first {S}teklov eigenvalue,
  conformal geometry, and minimal surfaces}, Adv. Math. \textbf{226} (2011),
  no.~5, 4011--4030. \MR{2770439}

\bibitem{Sim84}
Leon Simon, \emph{Lectures on geometric measure theory}, Centre for
  Mathematical Analysis, Australian National University, Canberra, 1984.

\bibitem{WanXiaZha25}
Guofang Wang, Chao Xia, and Xuwen Zhang, \emph{Monotonicity formulas for
  capillary surfaces}, Journal de Mathématiques Pures et Appliquées
  \textbf{204} (2025), 103802.
  \end{thebibliography}

\providecommand{\bysame}{\leavevmode\hbox to3em{\hrulefill}\thinspace}
\providecommand{\MR}{\relax\ifhmode\unskip\space\fi MR }
\providecommand{\MRhref}[2]{%
  \href{http://www.ams.org/mathscinet-getitem?mr=#1}{#2}
}
\providecommand{\href}[2]{#2}

\end{document}